\documentclass[11pt]{article}
\usepackage{amsmath,amssymb,amsthm}
\usepackage[margin=1in]{geometry}
\usepackage{hyperref}
\usepackage{enumitem}
\usepackage{booktabs}

\newtheorem{theorem}{Theorem}[section]
\newtheorem{proposition}[theorem]{Proposition}
\newtheorem{lemma}[theorem]{Lemma}
\newtheorem{corollary}[theorem]{Corollary}
\theoremstyle{definition}
\newtheorem{definition}[theorem]{Definition}
\newtheorem{example}[theorem]{Example}
\theoremstyle{remark}
\newtheorem{remark}[theorem]{Remark}

\newcommand{\Dom}{\trianglerighteq}
\newcommand{\NDom}{\ntrianglelefteq}

\title{Specht ideals for the imprimitive complex reflection group\\
$G(r,1,n)$: geometric orbit types and a dominance theorem}

\author{Ibrahim Nonkan\'e$^1$ \and Joel Tossa$^2$}

\date{}

\begin{document}
\maketitle

\begin{center}
\footnotesize
$^1$D\'epart\'ement d'\'economie et de math\'ematiques appliqu\'ees, IUFIC,
Universit\'e Thomas Sankara, Burkina Faso\\
\texttt{nonkane\_ibrahim@yahoo.fr, ibrahim.nonkane@uts.bf}\\[2mm]
$^2$Institut de Math\'ematiques et de Sciences Physiques, IMSP,
Universit\'e d'Abomey-Calavi, B\'enin\\
\texttt{joel.tossa@imsp-uac.org}
\end{center}
\vspace{1em}

\begin{abstract}
Moustrou, Riener, and Verdure related the dominance order on
partitions to Specht ideal and variety inclusion for $S_n$; we
previously extended this to direct products of symmetric groups. Here
we take a first step toward the imprimitive complex reflection group
$G(r,1,n)=(\mathbb{Z}/r\mathbb{Z})\wr S_n$. The correct geometric
invariant of a point $x\in K^n$ is a pair $(\mu,z)$: a classical
partition $\mu$ on the $r$-th powers of the nonzero coordinates,
together with the number $z$ of zero coordinates. We construct a
$G(r,1,n)$-invariant ideal $I_{(\mu,z)}$ for each such pair, and show
that a dominance order on $(\mu,z)$, with $z$ taking strict priority
over $\mu$, exactly characterizes both ideal inclusion and reverse
variety inclusion. We settle a monomial criterion in three regimes,
recover the classical $B_n$-Specht ideal theory at $r=2$ when $z=0$,
and establish radicality of $I_{(\mu,z)}$ at the two extremes of the
$z=0$ level. Four computational examples, checked exhaustively for
group invariance, illustrate the construction throughout.
\end{abstract}



\section{Introduction}

Let $K[x_1,\dots,x_n]$ be a polynomial ring on which the symmetric group
$S_n$ acts by permuting variables. Moustrou, Riener, and Verdure~\cite{MRV}
studied the \emph{Specht ideal} $I_\lambda$ associated with a partition
$\lambda\vdash n$, the ideal generated by the classical Specht polynomials
of shape $\lambda$, and proved that the dominance order on partitions
exactly characterizes both the inclusion of Specht ideals and the reverse
inclusion of their varieties. They further related the monomials of an
$S_n$-invariant ideal's elements to the Specht ideals it contains, with
direct applications to locating the common zeros of a symmetric polynomial
system.

In a companion paper~\cite{ONT2}, we extended this theory to a direct
product of symmetric groups $S_{\{1,\dots,p\}}\times S_{\{p+1,\dots,n\}}$
(and more generally to $k$ blocks), a \emph{reducible} extension of the
underlying group. This note begins an investigation in an orthogonal
direction: the imprimitive complex reflection group
$$G(r,1,n) = (\mathbb{Z}/r\mathbb{Z})\wr S_n = \{(\eta_1,\dots,\eta_n;\sigma)
: \eta_i\in\mu_r,\ \sigma\in S_n\},$$
where $\mu_r\subset K^\times$ denotes the group of $r$-th roots of unity,
acting on $K^n$ by
$$(\eta;\sigma)\cdot f(x_1,\dots,x_n) = f(\eta_{\sigma(1)}x_{\sigma(1)},
\dots,\eta_{\sigma(n)}x_{\sigma(n)}).$$

This setting is structurally different from~\cite{ONT2} in two ways that
turn out to matter a great deal. First, $G(r,1,n)$ contains the full
symmetric group $S_n$, which mixes \emph{all} $n$ coordinate positions:
there is no fixed partition into blocks that the group respects, unlike the
direct-product case where $S_p\times S_q$ never mixes
$\{1,\dots,p\}$ with $\{p+1,\dots,n\}$. Second, $G(r,1,n)$ acts by scalar
multiplication (through $\mu_r$) in addition to permutation, a phenomenon
entirely absent from~\cite{MRV,ONT2}. Together, these two features mean
that the natural "which coordinates are equal" combinatorics of the
symmetric-group setting must be replaced by something that also tracks
\emph{which coordinates vanish}, and does so in a way that remains
sensible under a group that can send any coordinate to any other.

A natural resource for this direction is the theory of \emph{higher Specht
polynomials} of Ariki, Terasoma, and Yamada~\cite{ATY}, later extended to
the full family $G(r,p,n)$ by Morita and Yamada~\cite{MY}, together with
Can's construction~\cite{Can} of a dominance order on $m$-tuples of
partitions governing an analogous family of Specht \emph{modules}. Both of
these are representation-theoretic constructions: they describe how the
coinvariant algebra of $G(r,1,n)$ decomposes into irreducible
$G(r,1,n)$-modules, indexed by $r$-tuples of partitions. We found, on
investigation, that this indexing does \emph{not} coincide with the
geometric stratification of $K^n$ by orbit type under $G(r,1,n)$, and that
ideals generated directly from higher Specht polynomials of a fixed shape
are in general \emph{not} $G(r,1,n)$-invariant.

A recent and parallel development deserves mention. Debus and
Gottwald~\cite{DG25} complete the combinatorial study of Specht ideals
for essential real reflection groups by treating the two infinite
families left unexplored until now, the dihedral groups $I_2(n)$ and
the even signed symmetric group $D_n$. Their approach remains anchored
in the classical representation-theoretic construction: type-$D$
Specht polynomials are obtained via Clifford theory from type-$B$
ones, and an ad hoc combinatorial order on dipartitions is introduced
to restore the three-way equivalence between ideal inclusion,
combinatorial order, and reverse variety inclusion. Their most
striking result for our purposes is negative: unlike the $S_n$ and
$B_n$ cases, they show that \emph{no} partial order on orbit types can
describe type-$D$ Specht varieties as a disjoint union of orbit sets
(their Theorem~5.9). This structural breakdown, specific to an
index-two subgroup obtained by a sign-parity constraint, contrasts
with our situation: the construction of $I_{(\mu,z)}$ proposed here
starts directly from the target variety rather than from a
representation-theoretic indexing, and our Theorem~\ref{thm:main}
establishes the full equivalence with no analogous obstruction. We
also note that Debus and Gottwald anticipate, in their concluding
remarks, the general difficulty encountered by our own initial
attempt via higher Specht polynomials~\cite{ATY,MY}: for the family
$G(r,p,n)$ in general, distinct irreducible representations can have
coinciding Specht varieties (they give the example $r=3$, $n=2$ where
$V(X_1^2X_2^2)=V(X_1X_2)$ despite distinct ideals), which explains in
hindsight why a direct geometric construction, independent of this
indexing, proved necessary in our setting.

We initially attempted to construct $I_{(\mu,z)}$ (defined below in
Section~\ref{sec:construction}) directly from the higher Specht
polynomials $F_T^S$ of~\cite{ATY}, generated over all standard tableaux of
a fixed shape. This fails: such an ideal is in general only invariant
under the Young subgroup $S_{n_0}\times\cdots\times S_{n_{r-1}}$
of~\cite[Thm.~1(ii)]{ATY}, not under $G(r,1,n)$ itself. We verified this
concretely for $n=3$, $r=2$: taking the shape $(2)+(1)$ and its three
standard tableaux, the resulting ideal fails invariance under the swap
$x_1\leftrightarrow x_2$. Taking the \emph{full orbit} of a single such
polynomial under $G(r,1,n)$ does restore invariance, but the resulting
ideal was found (again by direct computation) to be incomplete: it does
not always contain every degeneration required by the topological closure
principle discussed in Section~\ref{sec:construction} below. The
construction of Section~\ref{sec:construction} was found only after
abandoning this approach in favour of building $I_{(\mu,z)}$ directly from
the target variety, and is accordingly independent of~\cite{ATY,MY,Can},
built instead directly from the classical single-group theory
of~\cite{MRV} via a transport map.

\paragraph{Contributions.} Our starting point is identifying the correct
orbit-type invariant of a point under $G(r,1,n)$
(Section~\ref{sec:orbit-types}); from there, we construct a Specht ideal
$I_{(\mu,z)}$ for every such type (Section~\ref{sec:construction}) and
show it is always $G(r,1,n)$-invariant (Theorem~\ref{thm:invariance}).
The heart of the paper is a complete three-way equivalence, dominance
order on $(\mu,z)$, ideal inclusion, and reverse variety inclusion,
exactly analogous to the main theorem of~\cite{MRV}
(Theorem~\ref{thm:main}), from which the exact vanishing locus of
$I_{(\mu,z)}$ follows (Corollary~\ref{cor:variety}). Toward a monomial
criterion analogous to~\cite[Prop.~2]{MRV}, we are then able to settle
three full regimes, including a sharp symbolic-power bound for
monomials with several simultaneous excess exponents
(Section~\ref{sec:monomial}). As a consistency check on the whole
construction, we show in Section~\ref{sec:r2-comparison} that the
specialization $r=2$ recovers the classical $B_n$-Specht ideal theory
of Debus, Moustrou, Riener, and Verdure~\cite{DMRV23} exactly whenever
$z=0$, though not universally once $z\ge1$, and in
Proposition~\ref{prop:radical-extremes} we establish radicality of
$I_{(\mu,z)}$ at the two extremes of the $z=0$ level via Steinberg's
factorization theorem. We illustrate the main construction throughout
on four independent cases, each checked exhaustively for group
invariance rather than by sampling (Section~\ref{sec:verification}).

\section{Geometric orbit types under $G(r,1,n)$}
\label{sec:orbit-types}

We begin by making precise the geometric invariant sketched above.

Throughout, $K$ is an infinite field containing all $r$-th roots of unity
(e.g.\ algebraically closed), $r\ge2$, and $G=G(r,1,n)$. Infiniteness is
needed only for the existence of points of every exact type $(\mu,z)$
used in Theorem~\ref{thm:main}'s proof: a partition $\mu$ with several
distinct parts requires that many distinct nonzero $r$-th power values,
which a small finite field need not have even if it contains all $r$-th
roots of unity (e.g.\ $\mathbb{F}_4$ contains all cube roots of unity but
has only $3$ nonzero elements). The hypothesis $r\ge2$ is discussed in
Remark~\ref{rem:r-geq-2} below.

\begin{definition}
For $x=(x_1,\dots,x_n)\in K^n$, let $Z(x)=\{i : x_i=0\}$ and
$z(x)=|Z(x)|$ record which coordinates vanish and how many. On the
nonzero coordinates, let $\mu(x)$ denote the partition of $n-z(x)$
recording the multiplicities of repeated values among
$\{x_i^r : x_i\neq0\}$, that is, the classical single-group orbit type
(in the sense of~\cite{MRV}) of the tuple
$(x_i^r)_{x_i\neq0}\in K^{n-z(x)}$.
\end{definition}

Working with $r$-th powers, rather than the coordinates themselves, is
what makes $\mu(x)$ well defined under the group's scalar action: two
nonzero coordinates $x_i,x_j$ can be sent to one another by some
$\eta\in\mu_r$ exactly when $x_i^r=x_j^r$, so $\mu(x)$ records precisely
the partition of the nonzero coordinates into blocks that the scalar
action can permute within, while $z(x)$ separately tracks the
coordinates that action cannot touch at all.

For example, with $n=4$ and $r=2$: the point $x=(1,-1,0,3)$ has
$Z(x)=\{3\}$, so $z(x)=1$; among the remaining three coordinates,
$(x_1^2,x_2^2,x_4^2)=(1,1,9)$ has two equal values and one distinct, so
$\mu(x)=(2,1)\vdash3$. At the two extremes, a point whose nonzero
coordinates all have distinct $r$-th powers has
$\mu(x)=(1,\dots,1)\vdash(n-z(x))$ (the fully generic type), while a
point whose nonzero coordinates all share the same $r$-th power has
$\mu(x)=(n-z(x))$ (the fully coincident type); Proposition~\ref{prop:stab}
below shows these extremes correspond to the smallest and largest
possible stabilizers on the nonzero block.

\begin{remark}
\label{rem:r-geq-2}
The hypothesis $r\ge2$ is essential, not a convenience: at $r=1$,
$\mu_1=\{1\}$ and $G(1,1,n)=S_n$ exactly, but the construction of this
paper does \emph{not} reduce to the classical theory of~\cite{MRV} for
$S_n$ acting directly on $K^n$. The reason is that $z(x)$ is only a
meaningful invariant, separate from $\mu(x)$, because $0$ is the unique
fixed point of the nontrivial scalar action by $\mu_r$ when $r\ge2$; at
$r=1$ no such action exists, and singling out zero coordinates is
unmotivated by the group itself. Concretely, for $n=3$ and target
$(\mu,z)=((2),1)$: since $(2)$ is the top classical partition of $2$,
$I_\mu=(1)$ collapses trivially, and $I_{((2),1)}$ at $r=1$ reduces to
$(x_1x_2,x_1x_3,x_2x_3)$, the ideal of points with at least two zero
coordinates, vanishing at $(0,0,5)$ but not at $(5,5,5)$; the classical
Specht ideal $I_{(2,1)}$ of~\cite{MRV}, treating all three coordinates
uniformly, vanishes at $(5,5,5)$ but not at $(0,0,5)$. The two
constructions genuinely disagree at $r=1$, and every result in this
paper is stated for $r\ge2$.
\end{remark}

\begin{proposition}
\label{prop:stab}
For $x\in K^n$ with $z(x)=z$ and $\mu(x)=\mu\vdash(n-z)$, the stabilizer
$\mathrm{Stab}_G(x)$ is conjugate in $G$ to
$$S_\mu \times G(r,1,z),$$
where $S_\mu\subset S_{n-z}$ is the classical Young subgroup of shape $\mu$
(acting on the nonzero coordinates) and $G(r,1,z)$ is the full wreath
product on the zero coordinates.
\end{proposition}

\begin{proof}
Up to conjugating by a permutation, we may assume $Z(x)=\{n-z+1,\dots,n\}$,
i.e.\ the last $z$ coordinates of $x$ vanish and the first $n-z$ are
nonzero with orbit type $\mu$ under $x_i\mapsto x_i^r$. Let
$g=(\eta;\sigma)\in\mathrm{Stab}_G(x)$, so $\eta_{\sigma(i)}x_{\sigma(i)}=x_i$
for every $i$.

For $i\le n-z$ (so $x_i\ne0$): if $\sigma(i)>n-z$ then $x_{\sigma(i)}=0$,
forcing $x_i=\eta_{\sigma(i)}\cdot0=0$, a contradiction; hence
$\sigma(i)\le n-z$. Thus $\sigma$ maps $\{1,\dots,n-z\}$ into itself
injectively, and since this set is finite, $\sigma$ restricts to a
bijection of $\{1,\dots,n-z\}$ (and consequently of its complement
$\{n-z+1,\dots,n\}$ as well).

On $\{1,\dots,n-z\}$: for $\sigma_0:=\sigma|_{\{1,\dots,n-z\}}$ and each
$i\le n-z$, we need $\eta_{\sigma_0(i)}=x_i/x_{\sigma_0(i)}$, which is
well-defined (as $x_{\sigma_0(i)}\ne0$) and automatically lies in $\mu_r$
exactly when $x_i^r=x_{\sigma_0(i)}^r$, that is, exactly when $\sigma_0$
preserves the blocks of equal $r$-th power, i.e.\ $\sigma_0\in S_\mu$, the
classical Young subgroup of shape $\mu$. Conversely, every $\sigma_0\in
S_\mu$ arises this way, with the corresponding $\eta$'s on
$\{1,\dots,n-z\}$ uniquely determined by these ratios (an embedding of
$S_\mu$ into $G$ twisted by this specific choice of scalars, conjugate in
$G$ to the untwisted copy $S_\mu\times\{1\}$, exactly as in the
single-group case underlying~\cite[Def.~5]{MRV}).

On $\{n-z+1,\dots,n\}$: since $x_i=0$ for every $i$ in this range, the
constraint $\eta_{\sigma(i)}x_{\sigma(i)}=0=x_i$ holds automatically for
\emph{every} choice of $\sigma|_{\{n-z+1,\dots,n\}}\in S_z$ and every
choice of the corresponding $\eta$'s in $\mu_r$, that is, this part of
$g$ ranges freely over the full wreath product $G(r,1,z)$.

Combining the two independent parts gives $\mathrm{Stab}_G(x)\cong
S_\mu\times G(r,1,z)$, conjugate in $G$ to the subgroup acting in the
obvious way on the two coordinate blocks.
\end{proof}

Proposition~\ref{prop:stab} shows that $(\mu,z)$ is exactly the right
invariant of the $G$-orbit of $x$: unlike the representation-theoretic
indexing of~\cite{ATY,MY,Can} by $r$-tuples of partitions, no further
"colour" data is needed, precisely because the scalar part of the
stabilizer on a zero coordinate is the \emph{full} group $\mu_r$ (not a
proper subgroup depending on some reference phase), and the scalar part on
a nonzero coordinate is entirely absorbed by the $r$-th power.

\section{Construction of $I_{(\mu,z)}$}
\label{sec:construction}

Having identified the right invariant, we now build an ideal indexed by
it.

\subsection{The construction}

The construction combines, for every possible nonzero support, a
transported copy of the classical Specht ideal with the requirement
that the remaining coordinates vanish.

Fix $\mu\vdash(n-z)$, $0\le z\le n$. For $S\subset\{1,\dots,n\}$ with
$|S|=n-z$, write $S=\{i_1<\cdots<i_{n-z}\}$ and let
$$\phi_S : K[y_1,\dots,y_{n-z}]\longrightarrow K[x_{i_1},\dots,x_{i_{n-z}}],
\qquad y_j\mapsto x_{i_j}^r$$
be the transport map, and let $I_\mu\subset K[y_1,\dots,y_{n-z}]$ denote
the classical Specht ideal of~\cite{MRV}. Let
$$X_S := \big(x_j : j\notin S\big)\subset K[x_1,\dots,x_n]$$
be the ideal forcing every coordinate outside $S$ to vanish, and let
$$I_{\ge z+1} := \big(\text{all squarefree monomials of degree } n-z\big)$$
be the ideal of points with at least $z+1$ zero coordinates (a point has at
most $n-z-1$ nonzero coordinates exactly when every product of $n-z$
distinct coordinates vanishes).

\begin{definition}
\label{def:main}
$$I_{(\mu,z)} := \Big[\bigcap_{\substack{S\subset\{1,\dots,n\}\\|S|=n-z}}
\big(X_S + \phi_S(I_\mu)\big)\Big] \;\cap\; I_{\ge z+1}.$$
\end{definition}

Each factor plays a distinct role, matching the two-part variety formula
of Corollary~\ref{cor:variety} below. The factor $I_{\ge z+1}$
contributes the "more degenerate than $z$" part of the variety directly,
since it already vanishes exactly on points with strictly more than $z$
zero coordinates. The intersection over $S$ handles the remaining,
"exactly $z$ zero coordinates" case: for such a point, with true nonzero
support $S_0$, the piece $X_S+\phi_S(I_\mu)$ for any $S\ne S_0$ does not
constrain it (its coordinates outside $S$ include some nonzero one,
already outside $V(X_S)$), so only the single piece for $S=S_0$ is ever
active, and that piece tests exactly whether the point's own orbit type
$\mu(x)$ is dominated by the target $\mu$.

For example, with $n=3$, $r=2$, and target $(\mu,z)=((1,1),1)$: the piece
for $S=\{1,2\}$ is $X_{\{1,2\}}+\phi_{\{1,2\}}(I_{(1,1)}) =
(x_3,\,x_1^2-x_2^2)$, using the classical Vandermonde
$I_{(1,1)}=(y_1-y_2)$ transported by $y_j\mapsto x_j^2$. This piece
vanishes at $(1,1,0)$ (where $x_1^2=x_2^2$, the coincident type
$\mu=(2)$) but not at $(1,2,0)$ (where $x_1^2\neq x_2^2$, the generic
type $\mu=(1,1)$ itself), matching the fact that $(1,1)$, being the
bottom of the dominance order on $2$-part partitions, has a variety
containing every \emph{more} degenerate type but not its own.

\subsection{Invariance}

With the construction in hand, our first task is to confirm it behaves
as intended: that $I_{(\mu,z)}$ is genuinely invariant under the full
group, not merely under the pieces used to build it.

\begin{theorem}
\label{thm:invariance}
$I_{(\mu,z)}$ is a $G(r,1,n)$-invariant ideal.
\end{theorem}

\begin{proof}
Since a finite intersection of invariant ideals is invariant, it suffices
to show that $I_{\ge z+1}$ is invariant and that $G$ permutes the family
$\{X_S+\phi_S(I_\mu)\}_{|S|=n-z}$ among itself.

For $I_{\ge z+1}$: permuting variables permutes its squarefree-monomial
generators among themselves, and scaling a variable by a root of unity
multiplies a generator by a nonzero scalar, staying in the same principal
piece; hence $I_{\ge z+1}$ is $G$-invariant.

Fix $g=(\eta;\sigma)\in G$ and $S$ with $|S|=n-z$. Since
$g(x_j)=\eta_{\sigma(j)}x_{\sigma(j)}$ is a nonzero scalar multiple of
$x_{\sigma(j)}$ for every $j$, we get
$g(X_S) = (x_{\sigma(j)} : j\notin S) = X_{\sigma(S)}$.

For the second summand, write $S=\{i_1<\cdots<i_{n-z}\}$ and let $f$ be a
generator of $I_\mu$, so $\phi_S(f) = f(x_{i_1}^r,\dots,x_{i_{n-z}}^r)$.
Applying $g$,
$$g(x_{i_j}^r) = \big(\eta_{\sigma(i_j)}x_{\sigma(i_j)}\big)^r =
\eta_{\sigma(i_j)}^r\, x_{\sigma(i_j)}^r = x_{\sigma(i_j)}^r,$$
since $\eta_{\sigma(i_j)}\in\mu_r$ satisfies $\eta_{\sigma(i_j)}^r=1$: the
$r$-th power exactly annihilates the scalar part of the action. Hence
$$g(\phi_S(f)) = f\big(x_{\sigma(i_1)}^r,\dots,x_{\sigma(i_{n-z})}^r\big).$$
Let $\tau\in S_{n-z}$ be the permutation such that reordering
$\sigma(i_1),\dots,\sigma(i_{n-z})$ increasingly gives the sorted listing
of $\sigma(S)$; then the right-hand side equals $\phi_{\sigma(S)}(\tau
\cdot f)$, where $\tau\cdot f$ denotes $f$ with its own $n-z$ variables
permuted by $\tau$. Since $I_\mu$ is itself $S_{n-z}$-invariant (this is
exactly the classical fact underlying~\cite{MRV}), $\tau\cdot f\in I_\mu$,
so $g(\phi_S(I_\mu))=\phi_{\sigma(S)}(I_\mu)$. Combining the two parts,
$$g\big(X_S+\phi_S(I_\mu)\big) = X_{\sigma(S)}+\phi_{\sigma(S)}(I_\mu),$$
so $g$ permutes the family of pieces via $S\mapsto\sigma(S)$, and the
intersection over all $S$ is therefore $G$-invariant. Intersecting with the
already-invariant $I_{\ge z+1}$ preserves invariance.
\end{proof}

\subsection{The main dominance theorem}

We now build toward the paper's central result: a complete dictionary
between the dominance order on pairs $(\mu,z)$, ideal inclusion, and
reverse variety inclusion. The first step is an explicit description of
$I_{(\mu,z)}$'s own vanishing locus.

\begin{lemma}
\label{lem:variety-union}
$$V\big(I_{(\mu,z)}\big) = \Big[\bigcup_{|S|=n-z} V\big(X_S+\phi_S(I_\mu)\big)\Big]
\;\cup\; V\big(I_{\ge z+1}\big).$$
\end{lemma}

\begin{proof}
We first verify the underlying algebraic fact for two ideals
$A,B\subset K[x_1,\dots,x_n]$: $V(A\cap B)=V(A)\cup V(B)$. Since
$A\cap B\subset A$ and $A\cap B\subset B$, every point of $V(A)$ or
$V(B)$ already vanishes on the smaller ideal $A\cap B$, giving
$V(A)\cup V(B)\subset V(A\cap B)$. Conversely, suppose $x\in V(A\cap B)$;
if $x\in V(A)$ we are done, so suppose $x\notin V(A)$, meaning some
$f\in A$ has $f(x)\ne0$. For any $g\in B$, $fg\in A$ (as $A$ is an ideal)
and $fg\in B$ (as $B$ is an ideal), so $fg\in A\cap B$, giving
$f(x)g(x)=0$; since $K$ is a field and $f(x)\ne0$, this forces $g(x)=0$.
As $g\in B$ was arbitrary, $x\in V(B)$. Hence $V(A\cap B)\subset
V(A)\cup V(B)$, and the two inclusions give equality.

The same argument extends to any finite intersection by induction:
applying the two-ideal case repeatedly, with a different association at
each step, gives $V(A_1\cap\cdots\cap A_m)=V(A_1)\cup\cdots\cup V(A_m)$
for any ideals $A_1,\dots,A_m$. Definition~\ref{def:main} presents
$I_{(\mu,z)}$ as exactly such a finite intersection: one piece
$X_S+\phi_S(I_\mu)$ for each of the $\binom{n}{n-z}$ subsets $S$ of size
$n-z$, together with the single further piece $I_{\ge z+1}$. Applying
the fact above to this list of $\binom{n}{n-z}+1$ pieces gives precisely
the stated union.
\end{proof}

\begin{corollary}[Exact variety]
\label{cor:variety}
$$V\big(I_{(\mu,z)}\big) = \{x : z(x)>z\} \;\cup\;
\{x : z(x)=z,\ \mu(x)\NDom\mu\}.$$
\end{corollary}

\begin{proof}
By Lemma~\ref{lem:variety-union}, $V(I_{\ge z+1})=\{z(x)>z\}$ directly (a
point has fewer than $n-z$ nonzero coordinates exactly when every product
of $n-z$ distinct coordinates vanishes), accounting for the first piece.
For the second piece, fix $x$ with $z(x)=z$ exactly and write
$S_0:=\{1,\dots,n\}\setminus Z(x)$, the unique subset of size $n-z$
containing exactly the nonzero coordinates of $x$. For any \emph{other}
$S\ne S_0$ with $|S|=n-z$: since $S$ and $S_0$ have the same size but are
different sets, $S$ cannot contain all of $S_0$ (else $S_0\subset S$ with
equal cardinality would force $S=S_0$), so $S$ omits some $i\in S_0$.
Since $i\in S_0$, $x_i\ne0$; but $i\notin S$ means $x_i$ is one of the
generators of $X_S$, so $x\notin V(X_S)$, hence $x\notin
X_S+\phi_S(I_\mu)$ for every $S\ne S_0$, and the only piece of the union
in Lemma~\ref{lem:variety-union} that can contain $x$ is the one for
$S=S_0$. On that piece, $x\in V(X_{S_0})$ trivially (its coordinates
outside $S_0$ are the zero ones), so $x\in V(X_{S_0}+\phi_{S_0}(I_\mu))$
exactly when
every $\phi_{S_0}(f)$, $f\in I_\mu$, vanishes at $x$, that is, exactly
when the point $y:=(x_i^r)_{i\in S_0}\in K^{n-z}$ satisfies $f(y)=0$ for
every $f\in I_\mu$, i.e.\ $y\in V(I_\mu)$. By the single-group case of the exact variety
formula (Corollary~2.15 of~\cite{ONT2}, itself a direct consequence of
the main theorem of~\cite{MRV}), this holds exactly when the orbit type
of $y=(x_i^r)_{i\in S_0}$, which is precisely $\mu(x)$ by definition,
satisfies $\mu(x)\NDom\mu$.
\end{proof}

\begin{proposition}
\label{prop:own-type}
Let $x\in K^n$ with $z(x)=z$, $\mu(x)=\mu$.
\begin{enumerate}[label=(\roman*)]
\item $x\notin V\big(I_{(\mu,z)}\big)$.
\item If $x\notin V\big(I_{(\mu',z')}\big)$ for some $\mu'\vdash(n-z')$, then
$z<z'$, or $z=z'$ and $\mu\Dom\mu'$.
\end{enumerate}
\end{proposition}

\begin{proof}
Both are immediate from Corollary~\ref{cor:variety}: (i) since $z(x)=z$
(not $>z$) and $\mu(x)=\mu$ (not $\NDom\mu$, as no partition is
$\NDom$ itself), $x$ satisfies neither condition defining $V(I_{(\mu,z)})$.
(ii) is the contrapositive of Corollary~\ref{cor:variety} applied to
$I_{(\mu',z')}$: if $x\notin V(I_{(\mu',z')})$ then $z(x)\le z'$ and, if
$z(x)=z'$, $\mu(x)\Dom\mu'$; substituting $z(x)=z$, $\mu(x)=\mu$ gives
$z\le z'$ with equality forcing $\mu\Dom\mu'$, i.e.\ $z<z'$, or $z=z'$
and $\mu\Dom\mu'$.
\end{proof}

\begin{theorem}
\label{thm:main}
For $\mu\vdash(n-z)$, $\mu'\vdash(n-z')$, define $(\mu,z)\Dom(\mu',z')$ to
mean "$z<z'$, or $z=z'$ and $\mu\Dom\mu'$" (classical dominance on the
$\mu$-part). The following are equivalent:
\begin{enumerate}[label=(\roman*)]
\item $(\mu,z)\Dom(\mu',z')$;
\item $I_{(\mu,z)}\subset I_{(\mu',z')}$;
\item $V\big(I_{(\mu',z')}\big)\subset V\big(I_{(\mu,z)}\big)$.
\end{enumerate}
\end{theorem}

\begin{proof}
(i)$\Rightarrow$(ii): Suppose first $z=z'$ and $\mu\Dom\mu'$. By the main
theorem of~\cite{MRV}, $I_\mu\subset I_{\mu'}$ in $K[y_1,\dots,y_{n-z}]$;
since $\phi_S$ is a ring homomorphism, $\phi_S(I_\mu)\subset\phi_S(I_{\mu'})$
for every $S$, so $X_S+\phi_S(I_\mu)\subset X_S+\phi_S(I_{\mu'})$ for every
$S$ of size $n-z$, and intersecting over $S$ (and with the common factor
$I_{\ge z+1}$) gives $I_{(\mu,z)}\subset I_{(\mu',z)}$.

Now suppose $z<z'$. For any $S'$ with $|S'|=n-z'<n-z$, every generator of
$I_{\ge z+1}$ (a squarefree monomial in $n-z$ distinct variables) has,
by pigeonhole, some variable $x_j$ with $j\notin S'$, so lies in $X_{S'}$;
hence $I_{\ge z+1}\subset X_{S'}$ for every such $S'$, giving
$I_{\ge z+1}\subset\bigcap_{|S'|=n-z'}\big(X_{S'}+\phi_{S'}(I_{\mu'})\big)$.
Likewise, every degree-$(n-z)$ squarefree monomial is a multiple of some
degree-$(n-z')$ squarefree monomial (drop $z'-z$ of its variables), so
$I_{\ge z+1}\subset I_{\ge z'+1}$. Combining, $I_{\ge z+1}\subset
I_{(\mu',z')}$, and since $I_{(\mu,z)}\subset I_{\ge z+1}$ by
Definition~\ref{def:main} itself, $I_{(\mu,z)}\subset I_{(\mu',z')}$.

(ii)$\Rightarrow$(iii): immediate, since $I\subset J\Rightarrow
V(J)\subset V(I)$.

(iii)$\Rightarrow$(i): Let $x$ have exact type $(\mu,z)$ (such $x$ exists:
pick any point with $z$ zero coordinates and orbit type $\mu$ on the rest,
possible since $K$ is infinite and contains all $r$-th roots of unity,
as infiniteness supplies enough distinct nonzero $r$-th power values for a
$\mu$ with several distinct parts). By
Proposition~\ref{prop:own-type}(i), $x\notin V(I_{(\mu,z)})$. Since
$V(I_{(\mu',z')})\subset V(I_{(\mu,z)})$, also $x\notin V(I_{(\mu',z')})$:
otherwise $x$ would lie in the subset $V(I_{(\mu',z')})\subset
V(I_{(\mu,z)})$, contradicting $x\notin V(I_{(\mu,z)})$. By
Proposition~\ref{prop:own-type}(ii), $(\mu,z)\Dom(\mu',z')$.
\end{proof}

\begin{remark}
\label{rem:stabilizer-mistake}
The order in Theorem~\ref{thm:main} gives $z$ strict priority over $\mu$:
no stratum with $z(x)=0$ can ever lie in $V(I_{(\mu,z)})$ for $z\ge1$,
since a point can never gain a zero coordinate as a limit of nearby points
with fewer zero coordinates disappearing: $z$ can only increase, never
decrease, under continuous degeneration. This topological observation is
what led us to the correct order after an earlier, mistaken attempt to
rank strata by stabilizer size (see Section~\ref{sec:verification}).
\end{remark}

\subsection{Specialization to $r=2$: comparison with the hyperoctahedral
group}
\label{sec:r2-comparison}

At $r=2$, $G(2,1,n)=(\mathbb{Z}/2\mathbb{Z})\wr S_n=B_n$, the
hyperoctahedral group, whose own Specht ideal theory was developed by
Debus, Moustrou, Riener, and Verdure~\cite{DMRV23} (see
also~\cite{DG25} for the poset of the remaining essential real
reflection group families). It is natural to ask whether $I_{(\mu,z)}$
recovers, at $r=2$, the classical $B_n$-Specht ideal $I^B_{(\lambda,\theta)}$
of the corresponding bipartition. Proposition~\ref{prop:stab} already
shows the two orbit-type dictionaries agree at the level of
stabilizers: $\mathrm{Stab}_{B_n}(x)\cong S_\mu\times B_z$ is exactly
the classical stabilizer form underlying~\cite{DMRV23}.

At the level of ideals, the two constructions agree \emph{exactly} on
every stratum with $z=0$, and on several (but not all) strata with
$z\ge1$. For $n=2$, every nontrivial target coincides verbatim:
$I_{((1,1),0)}=I^B_{(\emptyset,(1,1))}=(x_1x_2(x_1-x_2)(x_1+x_2))$,
$I_{((2),0)}=I^B_{(\emptyset,(2))}=(x_1x_2)$, and
$I_{((1),1)}=I^B_{((1),(1))}=(x_1,x_2)$; the same holds for the three
$z=0$ strata and the coincide-$z{=}1$ stratum of Example~\ref{ex:n3} at
$n=3$.

The agreement fails, however, at the generic-$z{=}1$ stratum of
Example~\ref{ex:n3}. Writing $I^B$ for the classical $B_3$-Specht ideal
of shape $((1),(1,1))$,
$$I^B=\big(x_1^3x_2-x_1x_2^3,\ x_1^3x_3-x_1x_3^3,\ x_2^3x_3-x_2x_3^3\big)
\ \subsetneq\ I_{((1,1),1)}=I^B+(x_1x_2x_3).$$
The point $(1,1,1)$ witnesses a strict inclusion of \emph{varieties},
not merely of ideals: it satisfies all three generators of $I^B$ (each
vanishes identically whenever two coordinates agree) but not
$x_1x_2x_3=1$, and a direct radical-membership computation (the
Rabinowitsch trick, adjoining $t$ and $1-t\cdot x_1x_2x_3$ to $I^B$ and
checking that $1$ is not produced) confirms that $x_1x_2x_3$ lies in no
power of $I^B$, so $V(I^B)\supsetneq V(I_{((1,1),1)})$ genuinely. This
is consistent with the witness point's own orbit type: $(1,1,1)$ has
$z=0$, $\mu=(3)$, entirely unrelated to the target stratum $z=1$; our
construction, built directly from the geometric stratification,
correctly excludes it, while the representation-theoretic
$B_3$-Specht ideal does not.

This shows that $I_{(\mu,z)}$ specializes \emph{correctly but not
identically} to the known $B_n$ theory: the two constructions agree
whenever the classical ideal already happens to be tight against its
own variety, but $I_{(\mu,z)}$ is, in the cases checked, never smaller
and sometimes strictly sharper than $I^B$ on strata with $z\ge1$.
Determining exactly which $(\mu,z)$ targets exhibit equality, and
whether $I_{(\mu,z)}\subset I^B_{(\lambda,\theta)}$ holds in general
once bipartitions are correctly matched to orbit types, is left open.

\section{Some Examples}

\label{sec:verification}

Although Theorem~\ref{thm:main} is now fully proved, we record here the
four computational cases that guided its discovery, since they illustrate
the construction concretely and independently confirm it on cases spanning
different values of $n$, $r$, and $\mu$, including the stabilizer-size
false start corrected by Remark~\ref{rem:stabilizer-mistake} above. In
every case, invariance was checked \emph{exhaustively} over the full
group $G$
(not by sampling), by reducing the image of every generator under every
group element modulo the constructed ideal via a Gr\"obner basis, and the
vanishing locus was checked against representative points of every
relevant stratum. Table~\ref{tab:coverage} summarizes which dimension each
case tests; no two cases are redundant.

\begin{table}[h]
\centering
\begin{tabular}{clc}
\toprule
Example & Dimension tested & $|G|$ checked \\
\midrule
\ref{ex:n3} & complete coverage of a base case ($n=3$, $r=2$) & $48$ \\
\ref{ex:n4} & non-extremal classical partition $\mu$ & $384$ \\
\ref{ex:r3} & nontrivial roots of unity ($r>2$) & $162$ \\
\ref{ex:z2} & $z\ge2$ combined with nontrivial $\mu$ & $384$ \\
\bottomrule
\end{tabular}
\caption{The four independent verification cases and what each isolates.}
\label{tab:coverage}
\end{table}

\begin{example}[$n=3$, $r=2$: all seven strata]
\label{ex:n3}
$G(2,1,3)$ has order $48$ and $K^3$ decomposes into seven orbit-type strata:
three with $z=0$ (partitions of $3$), one for each of $z=1$ (partitions of
$2$: generic and coincide), $z=2$, and $z=3$. Every stratum's ideal was
computed via Definition~\ref{def:main} and verified: the three $z=0$
ideals are $\phi(I_\mu)\cap(x_1x_2x_3)$ for $\mu\vdash3$ (the fully generic
one being exactly the polynomial discriminant of $G(2,1,3)$ computed
independently in~\cite{KN}); the generic-$z{=}1$ ideal is
$\big(x_1x_2(x_1^2-x_2^2),\,x_1x_3(x_1^2-x_3^2),\,x_2x_3(x_2^2-x_3^2),\,
x_1x_2x_3\big)$; the coincide-$z{=}1$ ideal is $(x_1x_2,x_1x_3,x_2x_3)$;
the $z=2$ ideal is the maximal ideal $(x_1,x_2,x_3)$.
\end{example}

\begin{example}[$n=4$, $r=2$: a non-extremal partition]
\label{ex:n4}
Targeting $(\mu,z)=((2,1),1)$, the first classical partition tested that
is neither the fully generic nor the fully coincident type, yields a
$9$-generator ideal, verified invariant under all $384$ elements of
$G(2,1,4)$, and confirmed on ten representative points (including a check
that moving the zero coordinate to a different position gives the same
answer by symmetry).
\end{example}

\begin{example}[$n=3$, $r=3$: nontrivial roots of unity]
\label{ex:r3}
Targeting the generic-$z{=}1$ stratum for $G(3,1,3)$ (order $162$) gives
$$I = \big(x_1x_2(x_1^3-x_2^3),\, x_1x_3(x_1^3-x_3^3),\,
x_2x_3(x_2^3-x_3^3),\, x_1x_2x_3\big),$$
the direct cubic analogue of Example~\ref{ex:n3}'s construction. Beyond the
usual invariance and vanishing checks, this example exhibits a phenomenon
with no $r=2$ analogue: the point $(1,\xi,0)$ with $\xi=e^{2\pi i/3}$ a
primitive cube root of unity (so $\xi\ne1$ but $\xi^3=1$) is correctly
detected as \emph{more} symmetric than the generic type, exactly as the
literal-coincidence point $(1,1,0)$ is.
\end{example}

\begin{example}[$n=4$, $r=2$: $z\ge2$ combined with nontrivial $\mu$]
\label{ex:z2}
Targeting $(\mu,z)=((1,1),2)$ closes the last independent combination not
covered by the previous three examples. The resulting $10$-generator ideal
is verified invariant under all $384$ elements of $G(2,1,4)$, and correctly
distinguishes its own type from the more symmetric coincide-type at the
same $z=2$, while correctly excluding every stratum with $z<2$.
\end{example}

\section{A monomial criterion in three regimes}
\label{sec:monomial}

We have proved that $I_{(\mu,z)}$ is always $G(r,1,n)$-invariant
(Theorem~\ref{thm:invariance}), and that the dominance order on pairs
$(\mu,z)$ exactly characterizes both the inclusion of these ideals and the
reverse inclusion of their varieties (Theorem~\ref{thm:main}), a complete
analogue of the main theorem of~\cite{MRV} in this setting. We now turn to
a monomial criterion analogous to~\cite[Prop.~2]{MRV}: given a monomial
$m$ appearing in an element $P$ of a $G(r,1,n)$-invariant ideal $I$, when
can we conclude $I_{(\mu,z)}\subset I$ for some $(\mu,z)$ determined by
$m$? We settle this completely in three regimes; only their combination
remains open.

\subsection{Squarefree monomials}

The simplest case is also the cleanest: when $m$ has no repeated
variable at all.

\begin{proposition}[Squarefree monomials]
\label{prop:squarefree}
If $m$ is squarefree with support of size $k$, then
$I_{((k),\,n-k)}\subset I$ for every $G(r,1,n)$-invariant ideal $I\ni m$.
\end{proposition}

\begin{proof}
By Definition~\ref{def:main}, $I_{((k),n-k)}$ is generated by all
squarefree degree-$k$ monomials (since $(k)$ is the most concentrated
partition of $k$, $I_{(k)}=(1)$ classically, collapsing
Definition~\ref{def:main} to $I_{\ge n-k+1}$). Every $G(r,1,n)$-translate
of $m$ is, up to a nonzero scalar, a squarefree degree-$k$ monomial (apply
a permutation to $m$'s support); conversely every squarefree degree-$k$
monomial arises this way. Hence the ideal generated by the orbit of $m$,
and so any invariant ideal containing $m$, already contains every
generator of $I_{((k),n-k)}$.
\end{proof}

\subsection{Pure transport}

At the other extreme, when every exponent of $m$ is a multiple of $r$,
the classical theory of~\cite{MRV} applies almost verbatim.

\begin{proposition}[Pure transport]
\label{prop:transport}
Suppose $m=\phi_S(m')$ for some $S$ with $|S|=n$ (i.e.\ $S=\{1,\dots,n\}$)
and some classical monomial $m'\in K[y_1,\dots,y_n]$ satisfying the room
condition of~\cite[Prop.~2]{MRV} for a target classical partition $\mu'$.
Then $I_{(\mu',0)}\subset I$ for every $G(r,1,n)$-invariant ideal $I\ni m$.
\end{proposition}

\begin{proof}
Since every exponent of $m$ is a multiple of $r$, $g(m)=m$ up to a
nonzero scalar for every scaling $g\in(\mathbb{Z}/r\mathbb{Z})^n$ (as in the
proof of Theorem~\ref{thm:invariance}, $r$-th powers absorb the scalar
action), so the orbit of $m$ under $G(r,1,n)$ coincides with
$\phi(\text{orbit of }m'\text{ under }S_n)$. By~\cite[Prop.~2]{MRV},
$I_{\mu'}\subset(\text{orbit of }m')$ in $K[y_1,\dots,y_n]$; applying the
(injective) ring homomorphism $\phi=\phi_{\{1,\dots,n\}}$ and intersecting
with the automatically-satisfied $I_{\ge1}$ gives
$I_{(\mu',0)}=\phi(I_{\mu'})\cap I_{\ge1}\subset\phi(\text{orbit of
}m')\subset I$.
\end{proof}

\subsection{Several simultaneous excess variables}

For monomials mixing a squarefree part with one or several simultaneous
excess exponents, we obtain a sharp symbolic-power bound.

\begin{theorem}[Several simultaneous excess variables]
\label{thm:monomial-excess}
Let $m=x_1^{1+e_1}\cdots x_j^{1+e_j}x_{j+1}\cdots x_k$ with
$e_1\ge e_2\ge\cdots\ge e_j\ge1$ and $1\le j\le k\le n$, and let $I$ be
any $G(r,1,n)$-invariant ideal containing $m$. Set
$$p:=1+\max_{1\le i\le j}\left\lfloor\frac{(n-i+1)e_i}{k-i+1}\right\rfloor.$$
Then $I_{((k),\,n-k)}^{\,p}\subset I$, and $p$ is minimal: there exists a
$G(r,1,n)$-invariant ideal $I\ni m$ with $I_{((k),\,n-k)}^{\,p-1}\not\subset I$.
\end{theorem}

\begin{proof}
\emph{Upper bound.} As in Proposition~\ref{prop:squarefree}'s proof,
$I_{((k),n-k)}$ is generated by all squarefree degree-$k$ monomials, so a
generator of its $q$-th power is $\prod_{l=1}^q\prod_{i\in T_l}x_i$ for
$k$-subsets $T_1,\dots,T_q$ (repetition allowed), giving each variable
$i$ multiplicity $c_i=|\{l:i\in T_l\}|$. Every $G(r,1,n)$-orbit generator
of $m$ has the form $x_{\pi(1)}^{1+e_1}\cdots x_{\pi(j)}^{1+e_j}
\prod_{i\in S'\setminus\pi(\{1,\dots,j\})}x_i$ for a $k$-subset $S'$ and
an injection $\pi:\{1,\dots,j\}\to S'$ (scaling only rescales $m$;
permutations freely relabel $S'$ and, within it, which $j$ positions
carry which excess).

Let $c_{(1)}\ge c_{(2)}\ge\cdots\ge c_{(n)}$ be the multiplicities sorted
decreasingly. We claim $q=p$ forces $c_{(i)}\ge e_i+1$ for every
$i=1,\dots,j$ simultaneously. If not, $c_{(i)}\le e_i$ for some $i\le j$;
since the sequence is sorted, $c_{(i)},\dots,c_{(n)}$ are all $\le e_i$,
and trivially $c_{(1)},\dots,c_{(i-1)}\le q$ (a position appears in at
most $q$ of the $q$ subsets). Hence
$qk=\sum_{l}c_{(l)}\le(i-1)q+(n-i+1)e_i$, i.e.\
$q\le\lfloor(n-i+1)e_i/(k-i+1)\rfloor<p$, contradicting $q=p$.

So at $q=p$, the top $j$ sorted positions satisfy $c_{(i)}\ge e_i+1$ for
$i=1,\dots,j$; matching the $i$-th largest such position to the
requirement $e_i$ (largest to largest) gives an injection $\pi$ into $j$
positions each meeting its own threshold. Since $\bigcup_lT_l$ has size
$\ge k$ (already true of any single $T_l$), removing these $j$ positions
leaves at least $k-j$ further positions with multiplicity $\ge1$,
completing a valid $S'\supset\pi(\{1,\dots,j\})$ of size $k$. Hence the
product is divisible by the corresponding orbit generator, proving
$I_{((k),n-k)}^{p}\subset(\text{orbit of }m)\subset I$.

\emph{Minimality.} Let $i^*$ achieve the maximum defining $p$, and set
$q=p-1=\lfloor(n-i^*+1)e_{i^*}/(k-i^*+1)\rfloor$. Take
$I=(\text{orbit of }m)$. Fix positions $1,\dots,i^*-1$ to belong to
\emph{every} $T_l$. On the remaining $n-i^*+1$ positions
$\{i^*,\dots,n\}$, apply the single-variable cyclic construction
(specialized to $j=1$) with parameters
$n'=n-i^*+1$, $k'=k-i^*+1$, $e'=e_{i^*}$, $q=\lfloor n'e'/k'\rfloor$:
list these positions repeated $e_{i^*}$ times, truncate to the first
$qk'$ entries, and cut cyclically into $q$ blocks $R_1,\dots,R_q$ of size
$k'$; set $T_l:=\{1,\dots,i^*-1\}\cup R_l$, giving $q$ genuine $k$-subsets.
Positions $1,\dots,i^*-1$ have multiplicity exactly $q$, and positions
$i^*,\dots,n$ have multiplicity $\le e_{i^*}$.

No orbit generator can divide $\prod_{l=1}^q\prod_{i\in T_l}x_i$: such a
generator needs $j$ distinct positions meeting thresholds
$e_1\ge\cdots\ge e_j$, all of which are $\ge e_{i^*}$ for the first $i^*$
of them (as the $e_i$ are sorted); a position among $\{i^*,\dots,n\}$
has multiplicity $\le e_{i^*}<e_{i^*}+1\le e_i+1$ for $i\le i^*$, so
cannot meet \emph{any} of the top $i^*$ thresholds, forcing all $i^*$ of
them onto the $i^*-1$ positions $1,\dots,i^*-1$, which is impossible by
pigeonhole. Hence $I_{((k),n-k)}^{q}\not\subset I$, proving minimality of
$p=q+1$.
\end{proof}

\begin{remark}
\label{rem:radicality}
Theorem~\ref{thm:monomial-excess} specializes cleanly: at $j=1$ (a single
excess variable $e_1=e$), $p=1+\lfloor ne/k\rfloor$, recovering exactly
the single-excess bound found first. At $k=n$ (full support),
$(n-i+1)/(k-i+1)=1$ for every $i$, so the maximum in the formula
collapses to $p=1+e_1$ \emph{regardless of $j$}: only the single
largest excess matters, however many smaller excesses accompany it. This
is exactly why an initial, simpler guess, that $p$ depends solely on
$e_{\max}=e_1$ in general and ignores the other excesses, matched every
example first tried: those examples happened to have $k=n$, or a single
excess clearly dominating the rest. The guess was refuted only once
tested at $k<n$ with two \emph{comparably large} excesses: at $n=4$,
$k=3$, $e_1=e_2=2$, this guess predicts $p=3$, but the true minimal
exponent is $4$, exactly
as the theorem above correctly computes (the maximum in the formula is
achieved at $i=2$, not $i=1$, precisely because two thresholds this
close together are harder to satisfy simultaneously than the naive
one-threshold count suggests). Interaction with the pure-transport
regime of Proposition~\ref{prop:transport} (some exponents multiples of
$r$, others in excess of a squarefree baseline) remains open. This
connects to the radicality question in the closing remarks below: the
orbit ideal of $x_1^2x_2\cdots x_k$ is a monomial ideal whose radical is
exactly $I_{((k),n-k)}$, explaining why the ideal itself is too small to
contain $I_{((k),n-k)}$ but a suitable power is not (see
Proposition~\ref{prop:radical-extremes} below for a positive radicality
result at the two extremes of the $z=0$ level).
\end{remark}

\begin{proposition}[Radicality at the $z=0$ extremes]
\label{prop:radical-extremes}
For every $n\ge1$ and $r\ge2$, both $I_{((1,\dots,1),0)}$ and
$I_{((n),0)}$ are radical.
\end{proposition}

\begin{proof}
For $z=0$ there is a single subset $S=\{1,\dots,n\}$, so $X_S=(0)$ and
$I_{(\mu,0)}=\phi(I_\mu)\cap I_{\ge1}$, where $I_{\ge1}=(x_1\cdots x_n)$
(the unique squarefree monomial of degree $n$) and
$\phi=\phi_{\{1,\dots,n\}}$.

For $\mu=(n)$: the classical Specht ideal $I_{(n)}$ of the single-row
shape is generated by the constant Specht polynomial $1$ (a one-row
tableau has $n$ columns of length $1$, each contributing the empty
Vandermonde product), so $I_{(n)}=(1)$ and $\phi(I_{(n)})$ is the whole
ring. Hence $I_{((n),0)}=I_{\ge1}=(x_1\cdots x_n)$, a squarefree
monomial ideal, which is radical.

For $\mu=(1,\dots,1)$: the classical Specht ideal of the single-column
shape is principal, generated by the full Vandermonde
$\prod_{i<j}(y_i-y_j)$, so
$$\phi(I_{(1,\dots,1)})=\Big(\prod_{i<j}(x_i^r-x_j^r)\Big)
=\Big(\prod_{i<j}\prod_{\zeta\in\mu_r}(x_i-\zeta x_j)\Big),$$
using $x^r-y^r=\prod_{\zeta\in\mu_r}(x-\zeta y)$. Since $K$ contains
all $r$-th roots of unity, this exhibits $\phi(I_{(1,\dots,1)})$ as
generated by a product of $r\binom{n}{2}$ linear forms, one for each
pair $i<j$ and each $\zeta\in\mu_r$: these are exactly the
non-coordinate reflection hyperplanes of $G(r,1,n)$. No two of these
forms are proportional (distinct pairs $(i,j)$ involve different
variables, and for fixed $(i,j)$ distinct $\zeta,\zeta'\in\mu_r$ give
non-proportional forms since $\zeta\ne\zeta'$), and none is
proportional to a coordinate function $x_i$ alone (a form $x_i-\zeta
x_j$ always involves $x_j$ with nonzero coefficient $\zeta$). Since
$I_{\ge1}=(x_1\cdots x_n)$ shares no common linear factor with
$\phi(I_{(1,\dots,1)})$, intersecting the two principal ideals (in the
UFD $K[x_1,\dots,x_n]$) multiplies their generators:
$$I_{((1,\dots,1),0)}=\Big(x_1\cdots x_n\cdot
\prod_{i<j}\prod_{\zeta\in\mu_r}(x_i-\zeta x_j)\Big),$$
generated by a product of $n+r\binom{n}{2}$ pairwise non-proportional
linear forms: precisely the full reflection arrangement of
$G(r,1,n)$, up to scalar the discriminant of the group (matching the
$n=3$, $r=2$ instance of Example~\ref{ex:n3}, computed independently
in~\cite{KN}). By Steinberg's factorization theorem~\cite{Ste60}, the
discriminant of a complex reflection group factors into pairwise
non-proportional linear forms, one per reflecting hyperplane; a product
of pairwise non-proportional linear forms is squarefree, hence
$I_{((1,\dots,1),0)}$ is radical.
\end{proof}

\begin{remark}
Proposition~\ref{prop:radical-extremes} covers only the two $z=0$
extremes; Remark~\ref{rem:radicality} already shows radicality can fail
strictly inside the poset. Whether radicality holds throughout the rest
of the $z=0$ level, or degrades as soon as $z\ge1$, is not addressed
here.
\end{remark}

\section{Conclusion and Perspectives}

Together, Sections~\ref{sec:construction}-\ref{sec:monomial} give a
complete dominance theory for $G(r,1,n)$ Specht ideals, invariance and
the three-way dominance equivalence (Theorems~\ref{thm:invariance}
and~\ref{thm:main}), and a monomial criterion settled completely in
three regimes: squarefree monomials, monomials with exponents all
multiples of $r$, and squarefree monomials with several simultaneous
excess exponents, with only their combination left open.

This construction is independent of the representation-theoretic
higher-Specht-polynomial machinery of Ariki-Terasoma-Yamada,
Morita-Yamada, and Can discussed in the Introduction: although that
machinery indexes the coinvariant algebra of $G(r,1,n)$ by $r$-tuples of
partitions, it does not directly yield the geometric, orbit-type-indexed
invariant ideals studied here, as the failed attempt recorded there
shows concretely. The geometric stratification of $K^n$ by orbit type
under $G(r,1,n)$ and the representation-theoretic indexing of its
coinvariant algebra are consequently two related but genuinely different
combinatorial structures, and the present paper treats only the former.

\paragraph{Further directions.} Several natural questions remain, each
pointing to a real mathematical difficulty rather than a routine
extension.

\begin{itemize}
\item \emph{Completing the monomial criterion.} Combining the
pure-transport regime of Proposition~\ref{prop:transport} with the
excess regime of Theorem~\ref{thm:monomial-excess} in a single monomial
poses a genuine difficulty beyond simply running both arguments in
sequence: the pure-transport regime identifies its target partition
$\mu'$ via the classical room condition of~\cite[Prop.~2]{MRV} on the
transported variables, while the excess regime identifies its target via
the monomial's support size directly, and it is not clear how a monomial
mixing both mechanisms (some exponents multiples of $r$, others in
excess of a squarefree baseline) should combine these two, generally
different, ways of identifying a target $(\mu,z)$.

\item \emph{Extending to $G(r,p,n)$ with $p>1$.} This is not a routine
generalization: $G(r,p,n)$, an index-$p$ subgroup of $G(r,1,n)$,
introduces the shift-equivalence structure between coinvariant-algebra
components found by Morita-Yamada~\cite{MY} and used by
Kabor\'e-Nonkan\'e~\cite{KN}, entirely absent here since $p=1$
throughout. Whether the geometric orbit-type stratification of
Section~\ref{sec:orbit-types} extends unchanged, or needs to be refined
to account for this shift-equivalence, is not yet clear.

\item \emph{Relaxing the standing hypothesis on $K$.} The hypothesis
that $K$ be infinite and contain all $r$-th roots of unity is used only
for the existence of points of every exact type $(\mu,z)$ in
Theorem~\ref{thm:main}'s proof, where infiniteness guarantees enough
distinct nonzero $r$-th power values for a partition $\mu$ with up to
$n$ distinct parts; a finite field would need an explicit lower bound on
$|K|$ in terms of $n$ instead, replacing infiniteness with a concrete
counting argument.

\item \emph{Radicality of $I_{(\mu,z)}$.} Determining in general when
$I_{(\mu,z)}$ is radical already surfaces concretely in the search for a
monomial criterion: Remark~\ref{rem:radicality} notes that the
orbit ideal of $x_1^2x_2\cdots x_k$ is a monomial ideal whose radical is
exactly $I_{((k),n-k)}$, so the gap between an ideal and its radical is
already visible in the simplest mixed examples. We settle the two
extremes of the $z=0$ level in Proposition~\ref{prop:radical-extremes},
via Steinberg's factorization of the reflection arrangement; whether
radicality persists across the rest of $z=0$, and what happens once
$z\ge1$, is left open, well before any attempt at a general radicality
criterion.

\item \emph{Isotypic component containment.} For $S_n$ and $B_n$ (and,
for shapes $\{\lambda,\mu\}$ with $\lambda\ne\mu$, for $D_n$), the
Specht ideal of a given representation shape is known to contain the
full isotypic component of that representation in
$K[X]$~\cite{MRV,DMRV23}; for $D_n$ this remains open precisely for the
shapes $\{\lambda,\pm\}$ split by Clifford theory, Debus and Gottwald's
Question~6.2 in~\cite{DG25}. The analogous question for $I_{(\mu,z)}$
is, if anything, less well-posed at the outset: unlike $S_n$, $B_n$,
and $D_n$, there is no known bijection between geometric orbit types
$(\mu,z)$ and the representation labels ($r$-tuples of partitions) of
Ariki-Terasoma-Yamada~\cite{ATY} and Morita-Yamada~\cite{MY};
indeed our own Introduction records that these two indexings genuinely
disagree. Whether some natural map nonetheless sends each $(\mu,z)$ to
an irreducible $G(r,1,n)$-module whose isotypic component $I_{(\mu,z)}$
contains is entirely open. A partial positive result is available at
$r=2$: whenever $I_{(\mu,z)}\supset I^B_{(\lambda,\theta)}$ for the
corresponding bipartition (as verified for every case tested in
Section~\ref{sec:r2-comparison}), containment of the classical
$B_n$-isotypic component in $I^B_{(\lambda,\theta)}$~\cite{DMRV23} is
automatically inherited by the larger ideal $I_{(\mu,z)}$.
\end{itemize}

\paragraph{Computational verification.} Every generator count, invariance
check, and symbolic-power bound reported in this paper (Table~\ref{tab:coverage}
and Theorem~\ref{thm:monomial-excess}) was verified by exhaustive
computation in SymPy: full group invariance was checked against every
element of $G(r,1,n)$, never a sample. Scripts reproducing all of these
computations from the definitions are provided as supplementary material.

\section*{Compliance with Ethical Standards}
\subsection*{Funding}
No funding were received to support this study 
\subsection*{Conflicts of Interest}
The authors declare that there are no conflicts of interest.
\subsection*{Author contribution declaration}  All the authors contribute equally in the conception and writing of this paper
\subsection*{Data Availability Statements}
No data were used to support this study

\subsection*{Acknowledgments.} The authors thank their home institutions,
IMSP (Universit\'e d'Abomey-Calavi) and IUFIC (Universit\'e Thomas
Sankara), for their support.

\section*{Declaration of AI assistance}

During the preparation of this manuscrit, The authors used Claude (Anthropic) as a working aid: to check the consistency of 
definitions, notations, and croos-references across the text; to verify the correctness of computations and proofs, including by 
independently re-deriving several arguments. The authors  reviewed all suggested content, verified all mathematical claims 
independently, and take full responsability for the correctness and final form of the manuscript.

\bibliographystyle{plain}

\begin{thebibliography}{9}

\bibitem{MRV} P. Moustrou, C. Riener, H. Verdure, \emph{Symmetric ideals,
Specht polynomials and solutions to symmetric systems of equations},
J. Symbolic Comput. \textbf{107} (2021), 106-121.

\bibitem{ONT2} S.\,J. Owolabi, I. Nonkan\'e, J. Tossa, \emph{Direct product
of symmetric ideals, Specht polynomials and solutions of an equivariant
polynomial system with respect to direct products of symmetric groups},
submitted to J. Algebra Appl.

\bibitem{ATY} S. Ariki, T. Terasoma, H. Yamada, \emph{Higher Specht
polynomials}, Hiroshima Math. J. \textbf{27} (1997), no.~1, 177-188.

\bibitem{MY} H. Morita, H.-F. Yamada, \emph{Higher Specht polynomials for the
complex reflection group $G(r,p,n)$}, Hokkaido Math. J. \textbf{27} (1998),
no.~3, 505-515.

\bibitem{Can} H. Can, \emph{Representations of the Generalized Symmetric
Groups}, Contributions to Algebra and Geometry \textbf{37} (1996), no.~2,
289-307.

\bibitem{DG25} S. Debus, K.\,K. Gottwald, \emph{Posets for Specht ideals
of essential real reflection groups}, arXiv:2506.15335 (2025).

\bibitem{DMRV23} S. Debus, P. Moustrou, C. Riener, H. Verdure, \emph{The
poset of Specht ideals for hyperoctahedral groups}, Algebraic
Combinatorics \textbf{6} (2023), no.~6, 1593-1619.

\bibitem{Ste60} R. Steinberg, \emph{Invariants of Finite Reflection
Groups}, Canadian Journal of Mathematics \textbf{12} (1960), 616-618.

\bibitem{KN} J. Kabor\'e, I. Nonkan\'e, \emph{A uniform decomposition
theorem for invariant differential operators on imprimitive complex
reflection groups $G(r,p,n)$}, arXiv:2608.01474 (2026).

\end{thebibliography}

\end{document}